\documentclass{amsart}
\usepackage{amsfonts}
\usepackage{amsmath}
\usepackage[arrow,matrix]{xy}
\usepackage{graphicx}
\usepackage{pifont}
\usepackage{mathrsfs}
\let\mod=\undefined

\DeclareMathOperator{\End}{End} \DeclareMathOperator{\soc}{soc}
\DeclareMathOperator{\Ext}{Ext} \DeclareMathOperator{\Tor}{Tor}
\DeclareMathOperator{\op}{op} 
\DeclareMathOperator{\add}{add} \DeclareMathOperator{\Hom}{Hom}
 \DeclareMathOperator{\ind}{ind}
 \DeclareMathOperator{\mod}{mod}
\DeclareMathOperator{\rad}{rad}

\DeclareMathOperator{\Tr}{Tr}

\def\gldim{\text{\rm gl.\,dim}\,}
\def\pd{\text{\rm pd}}
\def\id{\text{\rm id}}

\newtheorem{theorem}{Theorem}
\newtheorem{lemma}[theorem]{Lemma}

\numberwithin{equation}{section}
\title[Tilted algebras and short chains of modules]{Tilted algebras and short chains of modules}

\author{Alicja Jaworska}
\address{Faculty of Mathematics and Computer Science, Nicolaus Copernicus University, Chopina 12/18, 87-100 Toru\'n, Poland}
\email{jaworska@mat.uni.torun.pl}
\thanks{The research supported by the Research Grant N N201 269135 of the Polish Ministry of Science and Higher Education.}
\author{Piotr Malicki}
\address{Faculty of Mathematics and Computer Science, Nicolaus Copernicus University, Chopina 12/18, 87-100 Toru\'n, Poland}
\email{pmalicki@mat.uni.torun.pl}

\author{Andrzej Skowro\'nski}
\address{Faculty of Mathematics and Computer Science, Nicolaus Copernicus University, Chopina 12/18, 87-100 Toru\'n, Poland}
\email{skowron@mat.uni.torun.pl}

\dedicatory{Dedicated to Idun Reiten on the occasion of her 70th birthday.}

\begin{document}
\date{}
\maketitle
\begin{abstract}
We provide an affirmative answer for the question raised almost twenty years ago in \cite{[RSS]} concerning the characterization of tilted artin algebras 
by the existence of a sincere finitely generated module which is not the middle of a short chain.
\end{abstract}
\section{Introduction}
\noindent Let $A$ be an artin algebra over a commutative artin
ring $R$, that is, $A$ is an $R$-algebra (associative, with
identity) which is finitely generated as an $R$-module. We denote
by $\mod A$ the category of finitely generated right $A$-modules,
by $\ind A$ the full subcategory of $\mod A$ formed by the
indecomposable modules, and by $K_0(A)$ the Grothendieck group of
$A$. Further, we denote by $D$ the standard duality $\Hom_R(-,E)$
on $ \mod A$, where $E$ is a minimal injective cogenerator in
$\mod R$. For a module $X$ in $\mod A$ and its minimal projective
presentation $\xymatrix@C=13pt{P_1 \ar[r]^f &  P_0 \ar[r] &X
\ar[r] &0}$ in $\mod A$, the transpose $\Tr X$ of $X$ is the
cokernel of the map $\Hom_A(f, A)$ in $\mod A^{\op}$, where
$A^{\op}$ is the opposite algebra of $A$. Then we obtain the
homological operator $\tau_A=D\Tr$ on modules in $\mod A$, called
the Auslander-Reiten translation, playing a fundamental role in
the modern representation theory of artin algebras. A module $M$
in $\mod A$ is said to be sincere if every simple right $A$-module
occurs as a composition factor of $M$. Finally, following
\cite{[Bon]}, \cite{[HRi1]}, $A$ is a tilted  algebra if $A$ is an
algebra of the form $\End_H(T)$, where $H$ is a hereditary artin
algebra and $T$ is a tilting module in $\mod H$, that is,
$\Ext^1_H(T,T)=0$ and the number of pairwise nonisomorphic
indecomposable direct summands of $T$ is equal to the rank of
$K_0(H)$.

The aim of the article is to establish the following characterization of tilted algebras. 
\begin{theorem} \label{thm1}
An artin algebra $A$ is a tilted algebra if and only if there
exists a sincere module $M$ in $\mod A$ such that, for any module
$X$ in $\ind A$, we have $\Hom_A(X,M)=0$ or $\Hom_A(M,\tau_AX)=0$.
\end{theorem}

We note that, by \cite{[AR]}, \cite{[RSS]},  a sequence
$\xymatrix@C=13pt{X \ar[r] & M \ar[r] & \tau_AX}$ of nonzero
homomorphisms in a module category $\mod A$ with $X$ being
indecomposable is called a short chain, and $M$ the middle of this
short chain. Therefore, the theorem asserts that an artin algebra
$A$ is a  tilted algebra if and only if $\mod A$ admits a sincere
module $M$ which is not the middle of a short chain. Hence, the
theorem provides an affirmative answer for the question raised in
\cite[Section 3]{[RSS]}. We would like to mention that some
related partial results have been proved by C.M. Ringel
\cite[p.376]{[Ri1]} and \O. Bakke \cite{[Bak]}.

The tilted algebras play a prominent role in the modern
representation theory of algebras and have attracted much
attention (see  \cite{[ASS]}, \cite{[MS]}, \cite{[Ri1]},
\cite{[Ri2]}, \cite{[SS1]}, \cite{[SS2]} and their cited papers).
In particular, the following handy criterion for an algebra to be
tilted has been established independently in \cite{[L2]},
\cite{[S1]}: an indecomposable artin algebra $A$ is a tilted
algebra if and only if the Auslander-Reiten quiver $\Gamma_A$ of
$A$ admits a component $\mathcal{C}$  with a faithful section
$\Delta$ such that  $\Hom _A(X, \tau_AY)=0$ for all modules $X$
and $Y$ in $\Delta$. We refer also to \cite{[JMS]}, \cite{[L2]},
\cite{[S1]}, \cite{[S2]}, \cite{[S3]} for characterizations of
distinguished classes of tilted algebras via properties of
components of the Auslander-Reiten quivers and to \cite{[S4]} for
a homological characterization of tilted algebras. The tilted
algebras are also crucial for the classification of finite
dimensional selfinjective algebras of finite growth over an
algebraically closed field. Namely, the indecomposable algebras in
this class are, up to Morita equivalence, socle deformations of
the orbit algebras $\widehat{B}/G$ of the repetitive algebras
$\widehat{B}$ of tilted algebras $B$ of Dynkin and Euclidean types
and admissible infinite cyclic groups $G$ of automorphisms of
$\widehat{B}$ (see the survey article \cite{[S5]}). Further, the
tilted algebras allow also to recover selfinjective artin algebras
whose Auslander-Reiten quiver admits components of a prescribed
form. For example, it has been shown in \cite{[SY1]} (see also
\cite{[SY2]}, \cite{[SY3]}) that, if the Auslander-Reiten quiver
of an indecomposable selfinjective artin algebra $A$ admits a
generalized standard (in the sense of \cite{[S2]}) acyclic
component, then $A$ is Morita equivalent to a socle deformation of
an orbit algebra $\widehat{B}/G$ of the repetitive algebra
$\widehat{B}$ of a tilted algebra $B$ not of Dynkin type and an
admissible infinite cyclic group $G$ of automorphisms of
$\widehat{B}$.

The tilted algebras belong to a wider class of algebras formed by
the quasitilted algebras, introduced by D. Happel, I. Reiten and
S. O. Smal{\o}  in \cite{[HRS2]}, which are the endomorphism
algebras $\End_{\mathcal{H}}(T)$ of tilting objects $T$ in
hereditary abelian categories $\mathcal{H}$. It has been shown in
\cite[Theorem 2.3]{[HRS2]} that these are exactly the artin
algebras $A$ of global dimension at most $2$  and with every
module in $\ind A$ of projective dimension or injective dimension
at most $1$. Further, by a result of D. Happel and I. Reiten
\cite{[HRe]}, every quasitilted artin algebra is a tilted algebra
or a  quasitilted algebra of canonical type. Moreover, by a result
of H. Lenzing and A. Skowro\'{n}ski \cite{[LS]}, the quasitilted
algebras of canonical type are the artin algebras whose
Auslander-Reiten quiver admits a separating family of semiregular
(ray or coray) tubes. The key step in our proof of the theorem is
to show that the module category $\mod A$ of a quasitilted but not
tilted algebra $A$ does not admit a sincere module which is not
the middle of a short chain, applying results from \cite{[Bae]},
\cite{[LP]}, \cite{[LS]}, \cite{[M]}, \cite{[Ri1]}, \cite{[Ri3]}.

For background on the representation theory applied here we refer
to \cite{[ASS]}, \cite{[ARS]}, \cite{[Ri1]}, \cite{[SS1]},
\cite{[SS2]}.

\section{Preliminaries}

In this section we briefly recall some of the notions we will use
and present an essential ingredient of the proof of the main
theorem of this article. This is concerned with relationship
between semiregular tubes and sincere modules which are not the
middle of a short chain.

Let $A$ be an artin algebra over a commutative artin ring $R$. We
denote by $\Gamma_A$ the Auslander-Reiten quiver of $A$. Recall
that $\Gamma_A$ is a valued translation quiver whose vertices are
the isomorphism classes $\{X\}$ of modules $X$ in $\ind A$, the
valued arrows of $\Gamma_A$ describe minimal left almost split
morphisms with indecomposable domain and minimal right almost
split morphisms with indecomposable codomain, and  the translation
is given by the Auslander-Reiten translations $\tau_A=D\Tr$ and
$\tau^-_A=\Tr D$. We shall not distinguish between a module $X$ in
$\ind A$ and the corresponding vertex $\{X\}$ of $\Gamma_A$.
Following \cite{[S2]},  a component  $\mathcal{C}$ of $\Gamma_A$
is said to be generalized standard if $\rad^{\infty}_A(X,Y)=0$
 for all modules $X$ and $Y$ in $\mathcal{C}$, where
 $\rad^{\infty}_A$is the infinite Jacobson radical of $\mod A$.
 Moreover, two components $\mathcal{C}$ and $\mathcal{D}$ of
 $\Gamma_A$ are said to be orthogonal if $\Hom_A(X,Y)=0$ and
 $\Hom_A(Y,X)=0$ for all modules $X$ in $\mathcal{C}$ and $Y$ in
 $\mathcal{D}$. A family $\mathcal{C}=(\mathcal{C}_i)_{i \in I}$
 of components of $\Gamma_A$ is said to be (strongly) separating
 if the components in $\Gamma_A$ split into three disjoint
 families $\mathcal{P}^A$, $\mathcal{C}^A= \mathcal{C}$ and
 $\mathcal{Q}^A$ such that the following conditions are
 satisfied:
 \begin{itemize}
 \item[] (S1) $ \mathcal{C}^A$ is a sincere family of pairwise orthogonal
 generalized standard com-
 \item[] \hspace{7mm} ponents;
 \item[] (S2) $\Hom_A(\mathcal{Q}^A,
 \mathcal{P}^A)=0$, $\Hom_A(\mathcal{Q}^A, \mathcal{C}^A)=0$,
 $\Hom_A(\mathcal{C}^A,
 \mathcal{P}^A)=0$;
 \item [] (S3) any morphism from $\mathcal{P}^A$ to $\mathcal{Q}^A$ in
 $\mod A$ factors through $\add (\mathcal{C}_i)$ for 
 \item[] \hspace{7mm} any $i \in I$.
 \end{itemize}
We then say that $\mathcal{C}^A$ separates $\mathcal{P}^A$ from
$\mathcal{Q}^A$ and write
$$\Gamma_A=\mathcal{P}^A \vee
\mathcal{C}^A \vee \mathcal{Q}^A.$$
 A component $\mathcal{C}$ of
$\Gamma_A$ is said to be preprojective if $\mathcal{C}$ is acyclic
(without oriented cycles) and each module in $\mathcal{C}$ belongs
to the $\tau_A$-orbit of a projective module. Dually,
$\mathcal{C}$ is said to be preinjective if $\mathcal{C}$ is
acyclic and each module in $\mathcal{C}$ belongs to the
$\tau_A$-orbit of an injective module. Further, $\mathcal{C}$ is
called regular if $\mathcal{C}$ contains neither a projective
module nor an injective module. Finally, $\mathcal{C}$ is called
semiregular if $\mathcal{C}$ does not contain both a projective
module and an injective module. By a general result of S. Liu
\cite{[L0]} and Y. Zhang \cite{[Z]}, a regular component
$\mathcal{C}$ contains an oriented cycle if and only if
$\mathcal{C}$ is a stable tube, that is, an orbit  quiver
$\mathbb{ZA}_{\infty}/(\tau^r)$, for some $r \geq 1$. Important
classes of semiregular components with oriented cycles are formed
by the ray tubes, obtained from stable tubes by a finite number
(possibly empty) of ray insertions, and the coray tubes obtained
from stable tubes by a finite number (possibly empty) of coray
insertions (see \cite{[Ri1]}, \cite{[SS2]}).

The following characterizations of ray and coray tubes have been
established by S. Liu in \cite{[L1]}.
\begin{theorem}
Let $A$ be an artin algebra and $\mathcal{C}$ be a semiregular
component of $\Gamma_A$. The following equivalences hold.
\begin{enumerate}
\renewcommand{\labelenumi}{\rm(\roman{enumi})}
\item  $\mathcal{C}$ contains an oriented cycle but no injective module if and only if $\mathcal{C}$ is a~ray tube.
\item $\mathcal{C}$ contains an oriented cycle but no projective module if and only if $\mathcal{C}$ is a~coray tube.
\end{enumerate}
\end{theorem}

The following lemma will play an important role in the proof of
our main theorem.

\begin{lemma}\label{lem1}
Let $A$ be an algebra and $M$ a sincere module in $\mod A$ which
is not the middle of a short chain. Then the following statements
hold.
\begin{enumerate}
\renewcommand{\labelenumi}{\rm(\roman{enumi})}
\item $\Hom_A(M,X)=0$ for any $A$-module $X$ in $\mathcal{T}$, where $\mathcal{T}$ is an arbitrary ray tube of $\Gamma_A$ containing a projective module.
\item $\Hom_A(X,M)=0$ for any $A$-module $X$ in $\mathcal{T}$, where $\mathcal{T}$ is an arbitrary coray tube of $\Gamma_A$ containing an injective module.
\end{enumerate}
\end{lemma}

\begin{proof}
We shall prove only (i), because the proof of (ii) is dual. Assume
that $\Hom_A(M,X)\neq 0$ for some $A$-module $X$ from
$\mathcal{T}$. Since $\mathcal{T}$ admits at least one projective
$A$-module, we conclude that there exists a projective $A$-module
$P$ in $\mathcal{T}$ lying on an oriented cycle. Let $\Sigma$ be
the sectional path in $\mathcal{T}$ from infinity to $P$ and
$\Omega$ be the sectional path in $\mathcal{T}$ from $X$ to
infinity. Then $\Sigma$ intersects $\Omega$ and there is a common
module $Y$ of $\Sigma$ and $\Omega$ different from $X$ and $P$.
Denote by $Z$ the immediate successor of $Y$ on the sectional path
from $Y$ to $P$. Then $\tau_AZ$ is the immediate predecessor of
$Y$ on the sectional path from $X$ to $Y$. Note that all
irreducible maps corresponding to the arrows on rays in
$\mathcal{T}$ are monomorphisms. Therefore, considering the
composition of irreducible monomorphisms corresponding to arrows
of $\Omega$ forming its subpath from $X$ to $\tau_AZ$ and using
our assumption $\Hom_A(M,X)\neq 0$, we get $\Hom_A(M,\tau_AZ)\neq
0$. Further, since $M$ is sincere and not the middle of a short
chain, $M$ is faithful by \cite[Corollary 3.2]{[RSS]}. Hence there
is a monomorphism $A_A\to M^r$ for some positive integer $r$, so a
monomorphism $P \rightarrow M^r$, because $P$ is a direct summand
of $A_A$. Considering the composition of irreducible homomorphisms
corresponding to arrows of $\Sigma$ forming its subpath from $Z$
to $P$, we receive that $\Hom_A(Z,P)\neq 0$, by a result of R.
Bautista and S. O. Smal{\o} \cite{[BS]}. Hence, $\Hom_A(Z,M^r)\neq
0$, and consequently $\Hom_A(Z,M) \neq 0$. Summing up, we have in
$\mod A$ a short chain $Z\to M\to \tau_AZ$, a contradiction. This
finishes our proof.
\end{proof}

\section{Proof of Theorem \ref{thm1}}

Let $A$ be an artin algebra over a commutative artin ring $R$. We
may assume (without loss of generality) that $A$ is basic and
indecomposable.

Assume first that $A$ is a tilted algebra, that is,
$A=\End_{H}(T)$ for a basic indecomposable hereditary artin
algebra $H$ over $R$ and a multiplicity-free tilting module $T$ in
$\mod H$. Then the tilting $H$-module $T$ determines the torsion
pair $(\mathcal{F}(T), \mathcal{T}(T))$ in $\mod H$, with the
torsion-free part $\mathcal{F}(T)=\{X \in \mod H |
\Hom_H(T,X)=0\}$ and the torsion part $\mathcal{T}(T)=\{X \in \mod
H | \Ext^1_H(T,X)=0\}$, and the splitting torsion pair
$(\mathcal{Y}(T), \mathcal{X}(T))$ in $\mod A$, with the
torsion-free part $\mathcal{Y}(T)=\{Y \in \mod A|
\Tor^A_1(Y,T)=0\}$ and the torsion part $\mathcal{X}(T)=\{Y \in
\mod A| Y \otimes_A T=0\}$. Then, by the Brenner-Butler theorem,
the functor $\Hom_H(T,-): \mod H \to \mod A$ induces an
equivalence of $\mathcal{T}(T)$ with $\mathcal{Y}(T)$, and the
functor $\Ext^1_H(T,-): \mod H \to \mod A$ induces an equivalence
of $\mathcal{F}(T)$ with $\mathcal{X}(T)$ (see \cite{[BB]},
\cite{[HRi1]}). Further, the images $\Hom_H(T,I)$ of the
indecomposable injective modules $I$ in $\mod H$ via the functor
$\Hom_H(T,-)$ belong to one component $\mathcal{C}_T$ of
$\Gamma_A$, called the connecting component of $\Gamma_A$
determined by $T$, and form a faithful section $\Delta_T$ of
$\mathcal{C}_T$, with $\Delta_T$ the opposite valued quiver
$Q^{\op}_H$ of the valued quiver $Q_H$ of $H$. Recall that a full
connected valued subquiver $\Sigma$ of a component $\mathcal{C}$
of $\Gamma_A$ is called a section if $\Sigma$ has no oriented
cycles, is convex in $\mathcal{C}$, and intersects each
$\tau_A$-orbit of $\mathcal{C}$ exactly once. Moreover, the
section $\Sigma$ is faithful provided the direct sum of all
modules lying on $\Sigma$ is a faithful $A$-module. The section
$\Delta_T$ of the connecting component $\mathcal{C}_T$ of
$\Gamma_A$ has the distinguished property: it connects the
torsion-free part $\mathcal{Y}(T)$ with the torsion part
$\mathcal{X}(T)$, because every predecessor in $\ind A$ of a
module $\Hom_H(T,I)$ from $\Delta_T$ lies in $\mathcal{Y}(T)$ and
every successor of $\tau^-_A \Hom_H(T,I)$ in $\ind A$ lies in
$\mathcal{X}(T)$. Let $M_T$ be the direct sum of all modules lying
on $\Delta_T$. Then $M_{T}=\Hom_{H}(T,D(H))$ and is a sincere
$A$-module. Suppose $M_T$ is the middle of a short chain $X\to
M_T\to \tau_AX$ in $\mod A$. Then $X$ is a predecessor of an
indecomposable direct summand $M'$ of $M_T$ in $\ind A$ and
consequently $M_T\in \mathcal{Y}(T)$ forces $X\in \mathcal{Y}(T)$.
Hence $\tau_AX$ also belongs to $\mathcal{Y}(T)$, since
$\mathcal{Y}(T)$ is closed under predecessors in $\ind A$. In
particular, $\tau_AX$ does not lie on $\Delta_T$. Then
$\Hom_A(M_T,\tau_AX)\neq 0$ implies that there is an
indecomposable direct summand $M''$ of $M_T$ such that $\tau_AX$
is a successor of $\tau^{-1}_AM''$ in $\ind A$. But then
$\tau^{-1}_AM'' \in \mathcal{X}(T)$ forces that $\tau_AX \in
\mathcal{X}(T)$, because $\mathcal{X}(T)$ is closed under
successors in $\ind A$. Hence $\tau_AX$ is simultaneously in
$\mathcal{Y}(T)$ and $\mathcal{X}(T)$, a contradiction. Thus
$M_{T}$ is, as wanted,
a sincere $A$-module which is not the middle of a short chain.\\

Conversely, assume that $M$ is a sincere module in $\mod A$ which is not the middle of a short chain. Then, by \cite[Theorem 3.5 and Proposition 3.6]{[RSS]},
we have that $\gldim A\leq 2$ and $\pd_{A}X\leq 1$ or $\id_{A}X\leq 1$ for any indecomposable module $X$ in $\mod A$. Thus $A$ is a quasitilted  algebra by
the characterization given by D. Happel, I. Reiten and S. O. Smal{\o} in \cite[Theorem 2.3]{[HRS2]}.

We shall now assume that $A$ is not a tilted algebra and prove that this leads to a contradiction.

Let $A$ be a quasitilted algebra which is not tilted. Then by the
result of D. Happel and I. Reiten \cite{[HRe]}, $A$ is a
quasitilted algebra of canonical type. Hence, following
\cite{[LS]}, the Auslander-Reiten quiver $\Gamma_{A}$ of $A$ has a
disjoint union decomposition of the form
\[\Gamma_{A}=\mathcal{P}^{A}\vee\mathcal{T}^{A}\vee\mathcal{Q}^{A},\]
where $\mathcal{T}^{A}$ is a separating family of pairwise
orthogonal generalized standard semi\-regular (ray or coray)
tubes. Note that the separating family $\mathcal{T}^{A}$ satisfies
the following conditions:
$\Hom_{A}(\mathcal{T}^{A},\mathcal{P}^{A})=0$,
$\Hom_{A}(\mathcal{Q}^{A},\mathcal{T}^{A})=0$,
$\Hom_{A}(\mathcal{Q}^{A},\mathcal{P}^{A})=0$ and every
homomorphism $f:U\rightarrow V$ with $U$ in $\mathcal{P}^{A}$ and
$V$ in $\mathcal{Q}^{A}$ factorizes through a module $W$ from the
additive category $\add(\mathcal{T})$ of any tube $\mathcal{T}$ of
$\mathcal{T}^{A}$.

Let $N$ be an indecomposable direct summand of $M$. Since $N$ is
not the middle of a short chain, $N$ does not lie on a short cycle
$X\to N \to X$ in $\ind A$ \cite[Lemma 1]{[HL]}, and consequently
$N$ cannot belong to any stable tube of $\Gamma_{A}$. Thus, in
view of Lemma \ref{lem1}, we conclude that $M$ has no direct
summands in $\mathcal{T}^{A}$. Therefore $M=M_P\oplus M_Q$, where
$M_P$ is the maximal direct summand of $M$ from the additive
category $\add(\mathcal{P}^{A})$ of $\mathcal{P}^A$ and $M_Q$ is
the maximal direct summand of $M$ from the additive category
$\add(\mathcal{Q}^{A})$ of $\mathcal{Q}^A$. We claim that either
$M_P=0$ or $M_Q=0$. Assume that $M_P\neq 0$ and $M_Q\neq 0$.
There are three cases to consider (by the structure of the family $\mathcal{T}^{A}$).\\

\textbf{Case 1.} $\mathcal{T}^{A}$ is a separating family of
stable tubes. By \cite[Theorem 1.1]{[LP]}, it follows that $A$ is
a concealed canonical algebra. Recall that a concealed canonical
algebra is an algebra of the form $B=\End_{\Lambda}(T)$, where
$\Lambda$ is a canonical algebra in the sense of C. M. Ringel
\cite{[Ri3]} and $T$ is a multiplicity-free tilting module in the
additive category $\add(\mathcal{P}^{\Lambda})$, for the canonical
decomposition $\Gamma_{\Lambda}= \mathcal{P}^{\Lambda} \vee
\mathcal{T}^{\Lambda} \vee \mathcal{Q}^{\Lambda}$ of
$\Gamma_{\Lambda}$, with $\mathcal{T}^{\Lambda}$ the canonical
infinite separating family of stable tubes of $\Gamma_{\Lambda}$.
Then $\mathcal{P}^{A}$ is a family of components containing all
the indecomposable projective $A$-modules, $\mathcal{Q}^{A}$ is a
family of components containing all the indecomposable injective
$A$-modules and $\mathcal{T}^{A}$ separates $\mathcal{P}^{A}$ from
$\mathcal{Q}^{A}$. Let $M'$ be an indecomposable direct summand of
$M_P$ and $M''$ be an indecomposable direct summand of $M_Q$. In
this situation we have that $\Hom_{A}(M',\mathcal{Q}^{A})\neq 0$,
because for the injective hull $u:M'\rightarrow E_{A}(M')$ of
$M'$, $E_{A}(M')$ is contained in $\add(\mathcal{Q}^{A})$. Dually,
$\Hom_{A}(\mathcal{P}^{A},M'')\neq 0$, since for the projective
cover $h:P_{A}(M'')\rightarrow M''$ of $M''$, $P_{A}(M'')$ belongs
to $\add(\mathcal{P}^{A})$. Thus there are indecomposable
$A$-modules $X$ and $Y$ in an arbitrary given stable tube
$\mathcal{T}$ from $\mathcal{T}^{A}$ such that $\Hom_{A}(M',X)\neq
0$ and $\Hom_{A}(Y,M'')\neq 0$.

Let $\Sigma$ be the infinite sectional path in $\mathcal{T}$ which starts at $X$ and $\Omega$ be the infinite sectional path in $\mathcal{T}$ which
terminates at $Y$. Note that since
$\mathcal{T}$ is a stable tube, all morphisms corresponding to arrows of $\Sigma$ are monomorphisms, whereas all morphisms
which correspond to arrows of $\Omega$ are epimorphisms.
Further, $\Sigma$ intersects $\Omega$ and there is a common module $V$ different from $X$ and $Y$.
Denote by $Z$ the immediate successor of $V$ on the path $\Omega$. Then $\tau_AZ$ is the immediate predecessor of $V$ on the path $\Sigma$.
Thus there are a nonzero monomorphism from $X$ to $\tau_AZ$, which is a composition of irreducible monomorphisms, and
a nonzero epimorphism from $Z$ to $Y$ being a composition of irreducible epimorphisms. Therefore,
$\Hom_A(M', \tau_AZ)\neq 0$ and $\Hom_A(Z, M'')\neq 0$. In that way
we obtain a short chain $Z\rightarrow M\rightarrow\tau_A Z$ in $\mod A$, which contradicts the assumption imposed on $M$.\\

\textbf{Case 2.} There exists an indecomposable projective
$A$-module belonging to $\mathcal{T}^{A}$, but there are no
indecomposable injective modules in $\mathcal{T}^{A}$. Thus $A$ is
almost concealed canonical, by \cite[Corollary 3.5]{[LS]}, and
$\mathcal{Q}^A$ contains all indecomposable injective modules.
Recall that  it means that $A=\End_{\Lambda}(T)$ for a canonical
algebra $\Lambda$ and a tilting module $T$ from the additive
category $\add(\mathcal{P}^{\Lambda}\cup \mathcal{T}^{\Lambda})$
of $\mathcal{P}^{\Lambda}\cup \mathcal{T}^{\Lambda}$, where
$\Gamma_{\Lambda}=\mathcal{P}^{\Lambda} \vee \mathcal{T}^{\Lambda}
\vee \mathcal{Q}^{\Lambda}$ is the canonical decomposition of
$\Gamma_{\Lambda}$ with $\mathcal{T}^{\Lambda}$ the canonical
separating family of stable tubes.

Let $M'$ be an indecomposable direct summand of $M_P$ and $\mathcal T$ be a ray tube of the family $\mathcal{T}^{A}$ which contains a projective module.
Again, for an injective hull $M'\rightarrow E_{A}(M')$ of $M'$ in $\mod A$, we have $E_{A}(M')\in \add(\mathcal{Q}^{A})$.
Using now the separation property of $\mathcal{T}^A$ we conclude that $\Hom_A(M',X) \neq 0$
for a module $X$ in $\mathcal{T}$, which leads to a contradiction with Lemma \ref{lem1}. Therefore, $M_P=0$. Applying
dual arguments we show that, if $\mathcal{T}^{A}$ contains an indecomposable injective $A$-module but no indecomposable projective $A$-modules belong to
$\mathcal{T}^{A}$, then $M_Q=0$.\\

\textbf{Case 3.} There exists a ray tube $\mathcal{T}'$ in $\mathcal{T}^{A}$ containing a projective $A$-module and a coray tube $\mathcal{T}''$ in
$\mathcal{T}^{A}$ containing an injective $A$-module. In this case, $A$ is a semiregular branch enlargement of a concealed canonical algebra $B$
with respect to a separating family $\mathcal{T}^B$ of stable tubes in $\Gamma_B$
(see \cite[Theorem 3.4]{[LS]}). Moreover, $A$ admits quotient algebras $A_l$ and $A_r$ such that $A_l^{\op}$ is almost concealed canonical with
$\mathcal{P}^A=\mathcal{P}^{A_l}$ and $A_r$ is almost concealed canonical with $\mathcal{Q}^A=\mathcal{Q}^{A_r}$.
We note that $A_l$ is a branch coextension of $B$ with respect to $\mathcal{T}^B$ and $A_r$ is a branch extension of $B$
with respect to $\mathcal{T}^B$.

We start with the observation that the coray tube $\mathcal{T}''$
contains an injective module $I$ such that $I\slash \soc I$ has an
indecomposable direct summand $X$ which is a $B$-module lying in a
stable tube of $\mathcal{T}^B$. Since $M$ is faithful (see again
\cite[Corollary 3.2]{[RSS]}), there is an epimorphism
$M^{s}\rightarrow I\slash \soc I $, for some positive integer $s$,
and hence an epimorphism $(M_P)^{s}\rightarrow I\slash \soc I$,
because $\Hom_A(\mathcal{Q}^A, \mathcal{T}^A)=0$. Clearly, then
there is also an epimorphism $(M_P)^s \rightarrow X$. Choose a
simple $B$-module $S$ such that there is an epimorphism $X
\rightarrow S$. Hence there is an epimorphism $M_P \rightarrow S$.
Let $M'$ be a direct summand of $M_P$ satisfying the condition
$\Hom_{A}(M',S)\neq 0$. Since $S$ is a $B$-module, for an
injective hull $S\rightarrow E_{A}(S)$ of $S$ in $\mod A$, we
conclude that $E_{A}(S)$ is an injective $A_{r}$-module, and so
$E_{A}(S)\in \mathcal{Q}^{A}$. Therefore, we obtain that
$\Hom_{A}(M',E_{A}(S))\neq 0$ with $E_A(S)$ in $\mathcal{Q}^A$.
Similarly, we may take an indecomposable projective module $P$ in
the ray tube $\mathcal{T}'$ such that $\rad P$ has a direct
summand $Y$ which is a $B$-module. Let $T$ be a simple submodule
of $Y$. Since there is a monomorphism $A_A \rightarrow M^r$, we
get a monomorphism $P\rightarrow (M_Q)^r$, because
$\Hom_A(\mathcal{T}^A, \mathcal{P}^A)=0$, and consequently there
is an embedding $T \rightarrow (M_Q)^r $. Then
$\Hom_{A}(P_{A}(T),M'')\neq 0$ holds for some direct summand $M''$
of $M_Q$ and the $A_l$-module $P_A(T) \in \mathcal{P}^A$ such that
$P_{A}(T)\rightarrow T$ is a projective cover of $T$.

Since $\mathcal{T}^A$ separates $\mathcal{P}^A$ from $\mathcal{Q}^A$ in $\mod A$, any nonzero homomorphism
from $M'$ to $E_A(S)$ factorizes through $\mathcal{T'}$ and any
nonzero homomorphism from $P_A(T)$ to $M''$ has a factorization
through $\mathcal{T''}$. But this contradicts Lemma \ref{lem1}.

Summing up, we have shown that:
\begin{itemize}
    \item[$\bullet$] if $M_P\neq 0$, then $\mathcal{T}^{A}$ is a separating family of coray tubes (thus without projectives),
    \item[$\bullet$] if $M_Q\neq 0$, then $\mathcal{T}^{A}$ is a separating family of ray tubes (thus without injectives),
    \item[$\bullet$] either $M_P=0$ or $M_Q=0$.
\end{itemize}
By duality, we may assume that $M_P=0$. Then, $M=M_Q\in
\add(\mathcal{Q}^{A})$ and $\mathcal{T}^{A}$ is a separating
family of ray tubes. Equivalently, $A$ is an almost concealed
canonical algebra.

By general theory of quasitilted algebras of canonical type
(\cite{[LP]}, \cite{[LS]}), the class of almost concealed
canonical algebras divides into three classes: almost concealed
canonical algebras of Euclidean type, tubular type and wild type.
According to the assumption that $A$ is not a tilted algebra, thus
not of Euclidean type, we shall consider only the case of tubular
type and of wild type. \\

\textbf{(i)} Suppose first that $A$ is an almost concealed
canonical algebra of tubular type. Then, by \cite[Chapter
5]{[Ri1]} (also \cite{[Ri3]}), the shape of the Auslander-Reiten
quiver $\Gamma_A$ of $A$ is as follows
\[\Gamma_A= \mathcal{P}(A) \vee \mathcal{T}^A_0 \vee (\bigvee_{q \in \mathbb{Q}^+} \mathcal{T}^A_q) \vee
\mathcal{T}^A_{\infty} \vee \mathcal{Q}(A),\] where
$\mathcal{P}(A)$ is a preprojective component with a Euclidean
section, $\mathcal{Q}(A)$ is a preinjective component with a
Euclidean section, $\mathcal{T}^A_0$ is an infinite family of
pairwise orthogonal generalized standard ray tubes containing at
least one indecomposable projective $A$-module,
$\mathcal{T}^A_{\infty}$ is an infinite family of pairwise
orthogonal generalized standard coray tubes containing at least
one indecomposable injective $A$-module, and $\mathcal{T}^A_q$,
for $q$ in  the set of positive rational numbers $\mathbb{Q}^+$,
is an infinite family of pairwise orthogonal generalized standard
faithful stable tubes. Moreover, for $r<s$ in $\mathbb{Q}^+\cup
\{0,\infty\}$, we have $\Hom_A(\mathcal{T}^A_s,\mathcal{T}^A_r)=0$.

We claim that $M \in \add(\mathcal{Q}(A))$. Let $N$ be an indecomposable direct summand of $M$.
Observe that because $N$ does not lie on a short cycle
in $\mod A$, $N$ cannot belong to any stable tube from the  families
$\mathcal{T}^A_q$. Moreover, by Lemma \ref{lem1}, $N$ is neither in a ray tube
of $\mathcal{T}^A_0$ with an indecomposable projective $A$-module nor in a coray tube of
$\mathcal{T}^A_{\infty}$ with an indecomposable injective $A$-module.
Since $M\in \add(\mathcal{Q}^{A})$, we obtain that
$N \in \mathcal{Q}(A)$, and hence $M\in\add (\mathcal{Q}(A))$.\\

Let $I$ be an indecomposable injective $A$-module from the family $\mathcal{T}_{\infty}^{A}$. Since $M$ is a faithful module, there is an epimorphism
of the form $M^{s}\rightarrow I$ for some positive integer $s$. On the other hand, $\Hom_A(\mathcal{Q}(A), \mathcal{T}^A_{\infty})=0$ forces
$\Hom_A(M,I)=0$, a contradiction. This fact excludes the class of almost concealed canonical algebras of tubular type from the class of algebras
which admit a sincere finitely generated module which is not the middle of a short chain. \\

\textbf{(ii)} Let now $A$ be an almost concealed canonical algebra
of wild type. Then, by \cite{[LS]} and \cite{[M]}, $\Gamma_{A}$ is
of the form
 \[\Gamma_A= \mathcal{P}^A \vee \mathcal{T}^A \vee
\mathcal{Q}^A,\]
 where $\mathcal{P}^A$ consists of a unique preprojective component $\mathcal{P}(A)$ and an infinite family of components
obtained from components of the form $\mathbb{ZA}_{\infty}$ by a
finite number (possibly empty) of ray insertions, $\mathcal{Q}^A$
consists of a unique preinjective component $\mathcal{Q}(A)$ and
an infinite family of components obtained from components of the
form $\mathbb{ZA}_{\infty}$ by a finite number (possibly empty) of
coray insertions. By \cite{[M]},  $\mathcal{Q}(A)$ is a
preinjective component $\mathcal{Q}(C)$ for some wild concealed
algebra $C$, which is an indecomposable quotient algebra of $A$.
Thus $C=\End_{H}(T)$ for an indecomposable wild hereditary algebra
$H$ and a tilting $H$-module $T$ from the additive category $\add
(\mathcal{P}(H))$ of the preprojective component $\mathcal{P}(H)$
of $\Gamma_{H}$. Moreover, $\mathcal{Q}^{A}$ contains at least one
injective module which does not belong to $\mathcal{Q}(C)$. In
fact, there is an indecomposable injective $A$-module $I$ in
$\mathcal{Q}^A\backslash\mathcal{Q}(A)$ such that $I/ \soc I$ has
an indecomposable direct summand $X$ which is a regular
$C$-module. Denote by $\mathcal{D}$ the component of $\Gamma_A$
that  $I$ belongs to. Since $M$ is a faithful $A$-module, there is
an indecomposable direct summand $M_1$ of $M$ such that
$\Hom_A(M_1,I) \neq 0$. Clearly, $M_1$ belongs to a component, say
$\mathcal{C}$, from $\mathcal{Q}^A \backslash \mathcal{Q}(A)$
($\mathcal{C}$ may be equal to
$\mathcal{D}$).\\

Applying \cite[Theorem 6.4]{[M]}, we obtain that there exists an
indecomposable module $V$ in $\mathcal{C}$ such that the left cone
$(\rightarrow V)$ of $V$ in $\mathcal{C}$ (of all predecessors of
$V$ in $\mathcal{C}$) consists entirely of indecomposable regular
$C$-modules and the restriction of $\tau_A$ to $(\rightarrow V)$
coincides with $\tau_C$. We note also that, by tilting theory,
there is an equivalence $\xymatrix{\add(\mathcal{R}(H))
\ar[r]^(.49){\sim}& \add(\mathcal{R}(C))}$ of the additive
categories of regular modules
over $H$ and over $C$, induced by the functor $\Hom_H(T,-): \mod H \rightarrow \mod C $.\\

Let $Y$ be a quasi-simple regular $C$-module from the left cone
$(\rightarrow V)$ of $\mathcal{C}$. Then, by a result of D. Baer
\cite{[Bae]} (see also \cite[Chapter XVIII, Theorem 2.6]{[SS2]}),
there exists a positive integer $m_0$ such that $\Hom_A(X,
\tau_A^mY)=\Hom_C(X, \tau^m_CY)$ $\neq 0$ for all integers $m \geq
m_0$. Consider now the infinite sectional path $\Sigma$ in
$\mathcal{C}$ which terminates at $M_1$. Then there exists $m \geq
m_0$ such that the infinite sectional path $\Omega$ in
$\mathcal{C}$ which starts at $\tau^m_AY=\tau^m_CY$ contains a
module $\tau_A Z$ with $Z$ lying on $\Sigma$. In particular,
applying \cite{[BS]}, we get that $\Hom_A(Z,M_1) \neq 0$ and
consequently $\Hom_A(Z,M) \neq 0$. We shall show that also
$\Hom_A(M, \tau_AZ) \neq 0$. Since $\tau_AZ$ belongs to $\Omega$,
there exists in $\mod A$ a short exact sequence
\[0 \to \tau_A^mY \to \tau_AZ \to W \to 0 \]
with $\tau_AW$ on $\Omega$ (see \cite[Corollary 2.2]{[AS0]}), and
hence a monomorphism from $\tau^m_AY$ to $\tau_AZ$. Observe that
there are in $\mod A$ epimorphisms $M^s\to I$ and $I\to X$.
Further, since $\Hom_A(X, \tau^m_AY) \neq 0$, we have $\Hom_A(M^s,
\tau^m_AY) \neq 0$, and hence $\Hom_A(M^s, \tau_AZ) \neq 0$.
Therefore, we get $\Hom_A(M,\tau_AZ)\neq 0$. This shows that $M$
is the middle of a short chain in $\mod A$. Hence any almost
concealed canonical algebra of wild type does not admit a sincere
finitely generated module which is not the middle of a short
chain.

This finally contradicts the assumption that $A$ is not tilted, which  finishes the proof.

%

\end{document}